\documentclass[11pt,a4paper]{amsart}
\usepackage[english]{babel}
\usepackage{cite}
\usepackage{amssymb,enumerate,latexsym,hyperref}
\usepackage{color}
\usepackage[pagewise]{lineno}

\newcommand{\assign}{:=}
\newcommand{\infixand}{\text{ and }}
\newcommand{\nobracket}{}

\newenvironment{proof*}[1]{\noindent\textbf{#1\ }}{\hspace*{\fill}$\Box$\medskip}
\newtheorem{theorem}{Theorem}[section]
\newtheorem{lemma}[theorem]{Lemma}
\newtheorem{proposition}[theorem]{Proposition}
\newtheorem{corollary}[theorem]{Corollary}
\theoremstyle{remark}
\newtheorem{remark}[theorem]{Remark}

\begin{document}
    
\title[Global product structure]{On the global product structure for uniformly quasiconformal Anosov diffeomorphisms}

\author{jiesong zhang}
\address{School of Mathematical Sciences, Peking University, Beijing, 100871,
China}
\email{zhjs@stu.pku.edu.cn}

\begin{abstract}
 We prove that a transitive uniformly quasiconformal Anosov diffeomorphism with a two-dimensional stable or unstable distribution has a global product structure. As an application, we remove a topological assumption in a global rigidity theorem for uniformly quasiconformal Anosov diffeomorphisms of Kalinin-Sadovskaya {\cite{ks03}}. 
\end{abstract}

{\maketitle}

{\tableofcontents}

\section{Introduction}
A diffeomorphism $f : M \rightarrow M$ on a compact Riemannian manifold $M$ is called \textit{Anosov} if the tangent bundle $TM$ admits a $Df$-invariant splitting $E^s \oplus E^u$ and there exists a positive integer $k$ such that for any $x\in M$, 
\[
\|Df^k|_{E^s(x)}\|<1<\|(Df^k|_{E^u(x)})^{-1}\|^{-1}.
\]
An open conjecture that dates back to Anosov and Smale says that all Anosov diffeomorphisms are topologically conjugate to automorphisms of infra-nilmanifolds. In classification results toward this end \cite{bri77,bri80,bm80,fra68,ham14,ks03,man74,new70}, a key step is establishing the global product structure.

Two foliations $\mathcal{F}_1$ and $\mathcal{F}_2$ have \textit{global product structure} if every leaf $\mathcal{F}_1(x)$ intersects every leaf $\mathcal{F}_2(y)$ at a unique point $[x,y]$. An Anosov diffeomorphism has global product structure if the stable and unstable foliations have global product structure on the universal covering. 

All known examples of Anosov diffeomorphisms have global product structure, but it is widely open that whether all Anosov diffeomorphisms have global product structure. Brin showed that Anosov diffeomorphisms with a ``pinched" spectrum have global product structure \cite{bri77}. To be precise, Recall that every diffeomorphism $f$ of a compact Riemannian manifold $M$ induces a bounded linear operator $f_\ast$ in the space of continuous vector fields defined by
\[
(f_\ast v)(x) = Df_{f^{-1}(x)}(v(f^{-1}(x))).
\]
According to Mather \cite{mat69}, Anosov diffeomorphisms are characterized by the condition that the spectrum $\sigma$ of the complexification of $f_\ast$ does not intersect the unit circle, thus is contained in the interiors of two annuli with radii $0 < \lambda_1 < \lambda_2 < 1$ and $1 < \mu_2 < \mu_1 < \infty$. An Anosov diffeomorphism has a pinched spectrum if 
\[
1 + \frac{\log \mu_2}{\log \mu_1} > \frac{\log \lambda_1}{\log \lambda_2} \quad\text{or}\quad 1 + \frac{\log \lambda_2}{\log \lambda_1} > \frac{\log \mu_1}{\log \mu_2}.
\]
The ``pinching condition" considered by Brin is a uniform condition, which allows him to obtain a uniform bound on the distance of disks in the unstable(stable) leaves after iteration and to establish the global product structure. 

It is natural to consider pointwise conditions as they are more general, i.e. we only restrict the behavior of $D f|_{T M_x}$ at each single point $x \in M$. For example, we consider the uniformly quasiconformal systems. An Anosov diffeomorphism $f$ is called \textit{uniformly $u$-quasiconformal} if $\text{dim} E^u \geq 2$ and the quasiconformal distortion
\[ K^u (n, x) := \frac{\max \{ \| D f^n (v) \| : v \in E^u (x), \| v \| = 1 \}}{\min \{ \| D f^n (v) \| : v \in E^u (x), \| v \| = 1 \}} \]
is uniformly bounded for all $n \in \mathbb{Z}$ and $x \in M$. The notions of uniform $s$-quasiconformality are defined similarly. If $f$ is both uniformly $u$-quasiconformal and uniformly $s$-quasiconformal, then it is simply called \textit{uniformly quasiconformal}. 

The quasiconformal condition is a geometric condition that arises from the study of geodesic flows on hyperbolic manifolds \cite{kan93, sul81, yue96}. Generally speaking, the quasiconformal condition provides a dynamically invariant geometric structure (conformal structure) and often implies rigidity results. For example, Fang \cite{fan04, fan07, fan14}, Kalinin \cite{ks03}, and Sadovskaya \cite{ks03, sad05} established remarkable rigidity results for uniformly quasiconformal Anosov diffeomorphisms and Anosov flows. Butler and Xu \cite{bx18} observed the rigidity phenomenon in uniformly quasiconformal partially hyperbolic systems. Quasiconformal methods also naturally appear in the study of holomorphic Anosov systems \cite{ghy95} and holomorphic partially hyperbolic systems \cite{xz24}. For uniformly quasiconformal Anosov diffeomorphisms, we have:

\begin{theorem}[Kalinin-Sadovskaya {\cite{ks03}}]
\label{t2}
Let $f$ be a transitive $C^{\infty}$ uniformly quasiconformal Anosov diffeomorphism of a compact manifold $M$. If $f$ has global product structure, then $f$ is $C^{\infty}$-conjugate to an affine Anosov automorphism of a finite factor of a torus.
\end{theorem}
In \cite{ks03}, Kalinin and Sadovskaya establish global product structure for uniformly quasiconformal Anosov diffeomorphisms when $E^u$ and $E^s$ have dimension at least $3$. Their method is to show that the stable and unstable holonomy is conformal, then according to a classical theorem by Liouville, conformal maps on domains in $\mathbb{R}^d$ ($d \geq 3$) are Möbius transformations. Then they apply an argument developed by Ghys \cite{ghy95} to establish the global product structure (see Proposition \ref{ksb}). Their proof meets difficulties in dimension 2, as conformal maps on domains in $\mathbb{R}^2$ may not be Möbius transformations. 

In this note, we establish global product structure for uniformly quasiconformal Anosov diffeomorphisms with two-dimensional distributions:
\begin{theorem}
\label{t1}
Let $f$ be a transitive $C^{\infty}$ uniformly $u$-quasiconformal Anosov diffeomorphism of a compact manifold $M$. If $\dim E^u = 2$, then the stable holonomy is globally defined, i.e., defined on the entire unstable leaf. In particular, $f$ has global product structure. 
\end{theorem}
\begin{remark}
The argument of Brin \cite{bri77} cannot be adapted directly in the uniformly quasiconformal setting. As we only have pointwise estimation, $Df(x)$ may differ significantly when $x$ changes. Therefore, we cannot obtain a uniform bound of the radius of closed disks in $W^u(W^s)$ after iteration as Brin did.
\end{remark}
As a corollary, we are able remove an assumption in the global rigidity theorem of Kalinin-Sadovskaya \cite{ks03} and deduce that:
\begin{corollary}
\label{c1}
Let $f$ be a transitive $C^{\infty}$ uniformly quasiconformal Anosov diffeomorphism of a compact manifold $M$, then $f$ is $C^{\infty}$-conjugate to an affine Anosov automorphism of a finite factor of a torus.
\end{corollary}
The key ingredient in this paper lies in the introduction of the Schwarzian derivative to study uniformly quasiconformal systems. By doing so, we are able to apply the techniques developed in \cite{ghy95}, originally designed for the study of holomorphic Anosov systems, to demonstrate that the Schwarzian derivatives of the holonomy maps vanish and the holonomy maps are Möbius transformations in two-dimensional cases. Then, according to a theorem by Kalinin-Sadovskaya (Proposition \ref{ksb}), this implies that the holonomy maps are globally defined. We believe that the approach of computing the Schwarzian derivative has broader applications for studying quasiconformal dynamical systems.

{\medskip}

{\noindent}{\textbf{Organization of this paper. }} In Section 2, we revisit some geometric properties of uniformly quasiconformal Anosov diffeomorphisms. Section 3 is dedicated to defining the Schwarzian derivative for the holonomy maps and providing proofs for several key lemmas. Finally, in Section 4, we present the proof of Theorem \ref{t1}.

{\medskip}

{\noindent}{\textbf{Acknowledgment. }}The author would like to thank Shaobo Gan, Yuxiang Jiao, and Disheng Xu for discussions and for listening to the proof, thanks Boris Kalinin and Amie Wilkinson for reading the first draft and helpful comments, thanks the anonymous referee for many valuable comments and suggestions. 
{\medskip}

\section{Some geometric preliminaries}
In this section, we recall some properties of quasiconformal Anosov diffeomorphisms established in \cite{ks03,sad05}. We begin with a non-stationary linearization of $f$ along $W^u$:
\begin{proposition}[Proposition 3.3 in {\cite{sad05}}]\label{normal form}
Let $f$ be a $C^{\infty}$ uniformly $u$-quasiconformal Anosov diffeomorphism of a compact manifold $M$. Then for any $x \in M$, there is a family of $C^{\infty}$ diffeomorphisms $\Phi_x : W^u (x) \rightarrow E^u_x$ satisfying the following:
\begin{enumerate}
\item $\Phi_x (x) = 0$ and $D \Phi_x (x) = \text{Id}$;
\item $D f|_{E^u} \circ \Phi_x = \Phi_{f x} \circ f$;
\item The family of diffeomorphisms $\{ \Phi_x \}_{x \in M}$ varies continuously with $x$.
\end{enumerate}
Moreover, the family of diffeomorphisms $\{ \Phi_x \}_{x \in M}$ satisfying the conditions above is unique.
\end{proposition}

A conformal structure on $\mathbb{R}^d$, where $d \geq 2$, is a class of proportional inner products. The space $\mathcal{C}^d$ of conformal structures on $\mathbb{R}^d$ is identified with the space of real symmetric positive definite $d \times d$ matrices with determinant 1, which is isomorphic to $SL(d, \mathbb{R}) / SO(d, \mathbb{R})$. $GL(d, \mathbb{R})$ acts transitively on $\mathcal{C}^d$ via
\[ X [C] = (\det X^{\top} X)^{- \frac{1}{d}} X^{\top} C X \quad \text{where } X
   \in GL (d, \mathbb{R}) \infixand C \in \mathcal{C}^d . \]
For each $x \in M$, let $\mathcal{C}^u(x)$ be the space of conformal structures on $E^u(x)$. We obtain a continuous bundle $\mathcal{C}^u$ over $M$, with its fibers over $x$ being $\mathcal{C}^u(x)$. A continuous section of $\mathcal{C}^u$ is referred to as a continuous conformal structure. 
\begin{proposition}[Proposition 3.2 in \cite{sad05}]\label{sho0}Let $f$ be a transitive $C^{\infty}$ uniformly $u$-quasiconformal Anosov diffeomorphism of a compact manifold $M$, then there exists a continuous $f$-invariant conformal structure $\tau: M \to \mathcal{C}^u$.
\end{proposition}

For each $x \in M$ we extend the conformal structure $\tau(x)$ at $0 \in E^u(x)$ to all other points of $E^u(x)$ via translations. We denote this constant (translation-invariant) conformal structure on $E^u(x)$ by $\sigma(x)$. Since $\tau$ is $f$-invariant, $\sigma$ is $D f$-invariant. Let $A_x \in \text{SL}(2,\mathbb R)$ be a representative element of $\sigma(x)$, i.e., a linear map that take $\sigma(x)$ on $E^u(x)$ to the natural (Euclidean) conformal structure on $\mathbb{C} = \mathbb{R}^2$. The following proposition is essentially established in \cite{sad05}: 

\begin{proposition}[Lemma 3.1 and Theorem 1.4 in {\cite{sad05}}]\label{sho}
Let $f$ be a transitive $C^{\infty}$ uniformly $u$-quasiconformal Anosov diffeomorphism of a compact manifold $M$, then for every $x \in M$, $y \in W^s (x)$, and $U_x\subset W^u(x)$ such that the stable holonomy map $h^s_{xy}$ is well-defined on $U_x$, the map $\Phi_y \circ h^s_{x y} \circ \Phi_x^{-1} : \Phi_x(U_x) \rightarrow E^u_y$ is conformal with respect to $\sigma(x)$ and $\sigma(y)$. In particular, $h^s_{x y}$ is $C^\infty$ and
\[A_y \circ \Phi_y \circ h^s_{x y} \circ \Phi_x^{- 1} \circ A_x^{- 1}  : A_x \circ \Phi_x(U_x) \rightarrow \mathbb C\]
is a conformal (thus holomorphic) map on $A_x \circ \Phi_x(U_x) \subset \mathbb C$.
\end{proposition}
\begin{remark}\label{remark non unique}
  The choice of $A_x$ is not unique. A matrix $A'_x \in \text{SL}(2,\mathbb R)$ is a representative element of $\sigma(x)$ if and only if $A'_x =O_x A_x$ for some $O_x \in \text{SO}(2,\mathbb R)$. Therefore, $A_y \circ \Phi_y \circ h^s_{xy} \circ \Phi_x^{-1} \circ A_x^{-1}$ is always holomorphic for any choice of representative elements. Although these holomorphic functions depend on the choice of representative elements, we will show that they have the same Schwarzian derivative in the next section.
\end{remark}

\section{Schwarzian derivative}
Recall that the Schwarzian derivative of a holomorphic function $f : U \rightarrow \mathbb{C}$ is a holomorphic quadratic differential:
\[ S f  \assign \left( \left( \frac{f''' }{f' } \right) - \frac{3}{2}
   \left( \frac{f'' }{f' } \right)^2 \right) d z^2 . \]
The condition $Sf = 0$ holds if and only if $f$ belongs to $\text{PSL}(2, \mathbb{C})$ is a Möbius transform. For any holomorphic function $g:f(U) \to \mathbb C$, the Schwarzian derivative of the composition of holomorphic functions is given by:
\begin{equation}\label{e1}
     S (f \circ g)  = Sf\circ dg+Sg.
\end{equation}
In particular,
\begin{equation}\label{e2}
S (g_1 \circ f \circ g_2) = Sf\circ dg_2
\end{equation}
if both $g_1$ and $g_2$ are Möbius transformations. 

{\medskip}

Recall that by Proposition \ref{sho}, for every $x \in M, y \in W^s(x)$ and $U\subset \mathbb C$ such that $h^s_{x y}$ is well-defined on $\Phi_x^{- 1} \circ A_x^{- 1}(U)$, 
\[ A_y \circ \Phi_y \circ h^s_{x y} \circ \Phi_x^{- 1} \circ A_x^{- 1} :
   U \rightarrow \mathbb{C} \]
is conformal and holomorphic and maps $0$ to $0$. In particular, we can compute its Schwarzian derivative. For any $z \in \Phi_x^{- 1} \circ A_x^{- 1}(U)$ and $v \in E^u_z$, we define
\[ q_{x y}(v) \assign S (A_y \circ \Phi_y \circ h^s_{x y} \circ \Phi_x^{- 1}
   \circ A_x^{- 1})(A_x  v). \]
It follows from the definition that $q_{xy}(v)$ is continuous with respect to $x,y$ and $v$. As pointed out in Remark \ref{remark non unique}, $A_x$ and $A_y$ are not uniquely determined by $\sigma(x)$ and $\sigma(y)$, so we need to prove that $q_{x y}$ is well-defined.

\begin{lemma}
  \label{l1}$q_{x y}$ is well-defined, i.e., $q_{x y}$ does not depend on the choice of $A_x$ and $A_y$.
\end{lemma}

\begin{proof}
Let $A_x', A_y' \in \text{SL}(2, \mathbb{R})$ be another representative elements of $\sigma(x)$ and $\sigma(y)$ respectively, then $A_x' = O_x A_x$ and $A_y' = O_y A_y$ for some $O_x, O_y \in SO(2,\mathbb{R})$. For any $v \in E^u_x$, we have:
\begin{eqnarray*}
    &  & S (A_y' \circ \Phi_y \circ h^s_{x y} \circ \Phi_x^{- 1} \circ
    (A'_x)^{- 1})(A_x' v) \\
    & = & S (O_y \circ A_y \circ \Phi_y \circ h^s_{x y} \circ \Phi_x^{- 1}
    \circ A_x^{- 1} \circ O_x^{- 1})(O_x A_x v) \\
 (\text{by } (\ref{e2}))   & = & (S (A_y \circ \Phi_y \circ h^s_{x y}
    \circ \Phi_x^{- 1} \circ A_x^{- 1}) \circ O_x^{- 1})(O_x A_x v)  \\
    & = & S (A_y \circ \Phi_y \circ h^s_{x y} \circ \Phi_x^{- 1} \circ A_x^{-
    1}) (A_x v) .
  \end{eqnarray*}
Therefore, $q_{x y}$ does not depend on the choice of $A_x$ and $A_y$ and is well-defined.
\end{proof}
We then prove several properties of $q_{x y}$ which will be used later:

\begin{lemma}
  \label{l2}For any $x \in M$, $y \in W^s (x)$, and $v \in E^u_x$, we have:
  \[ q_{x y} (v) = q_{f^{- 1} x f^{- 1} y} \left( D f^{- 1}(v) \right) . \]
\end{lemma}

\begin{proof}
It follows the definitions that
  \[ A_z \circ D f \circ A_{f^{- 1} z}^{-1} \]
  is a conformal linear map on $\mathbb{C}$ for any $z \in M$. A direct calculation implies:
\begin{eqnarray*}
    q_{x y} (v)& = & S (A_y \circ \Phi_y \circ h^s_{x y} \circ \Phi_x^{- 1}
    \circ A_x^{- 1}) (A_x v)\\
    & = & S (A_y \circ \Phi_y \circ (f \circ h^s_{f^{- 1} x f^{- 1} y} \circ
    f^{- 1}) \circ \Phi_x^{- 1} \circ A_x^{- 1}) (A_x v)\\
   & = & S \left( A_y \circ D f \circ \Phi_{f^{- 1} y}
    \circ h^s_{f^{- 1} x f^{- 1} y} \circ \Phi_{f^{- 1} x}^{- 1} \circ 
    D f ^{- 1} \circ A_x^{- 1} \right) (A_x v)\\
    & = & S \left( (A_y \circ D f \circ A_{f^{- 1} y}^{-
    1}) \circ (A_{f^{- 1} y} \circ \Phi_{f^{- 1} y} \circ h^s_{f^{- 1} x f^{-1} y} \circ \Phi_{f^{- 1} x}^{- 1} \circ A_{f^{- 1} x}^{- 1}) \right.\\
    & & \left. \circ ( A_{f^{- 1} x}  \circ  D f^{- 1} \circ
    A_x^{- 1} )\right) (A_x v)\\
(\text{by (\ref{e2})})    & = & S (A_{f^{- 1} y} \circ \Phi_{f^{- 1} y} \circ h^s_{f^{- 1} x f^{-
    1} y} \circ \Phi_{f^{- 1} x}^{- 1} \circ A_{f^{- 1} x}) \left( A_{f^{- 1}
    x} \circ  D f^{- 1} (v) \right)\\
    & = & q_{f^{- 1} x f^{- 1} y} \left( D f^{- 1}(v) \right) .
\end{eqnarray*} 
\end{proof}

\begin{lemma}
  \label{l3}
  For any $x \in M$ and a simply connected region $U \subset W^s(x)$, the following equality holds for any $v \in E^u_x$ and $x' \in W^s(x)$:
  \[ \text{diam} \{ q_{x y}(v) : y \in U \} = \text{diam} \{ q_{x' y}(D h_{x x'}^s  (v)) : y \in U \}, \]
  where ``diam" denotes the diameter of a set in $\mathbb{C}$ with respect to the canonical metric.
\end{lemma}

\begin{proof}
  For any $y \in U$, we have
  \begin{eqnarray*}
    q_{x y}(v) & = & S (A_y \circ \Phi_y \circ h^s_{x y} \circ \Phi^{-1}_x \circ A_x^{-1}) (A_x v)\\
    & = & S ((\nobracket A_y \circ \Phi_y \circ h^s_{x' y} \circ \Phi_{x'}^{-1} \circ A_{x'}^{-1}) \circ (A_{x'} \circ \Phi_{x'} \circ h_{x x'}^s \circ \Phi_x^{-1} \circ A_x^{-1})) (A_x v) \nobracket\\
   (\text{by (\ref{e1})}) & = & S (A_y \circ \Phi_y \circ h^s_{x' y} \circ \Phi_{x'}^{-1} \circ A_{x'}^{-1}) (D_{0}(A_{x'} \circ \Phi_{x'} \circ h_{x x'}^s \circ \Phi_x^{-1}\circ A_x^{-1}) (A_x v)) + q_{x x'} (v)\\
    & = & q_{x' y}(D  h_{x x'}^s (v)) + q_{x x'} (v),
  \end{eqnarray*}
where $D_0$ denoting taking derivative at $0 \in \mathbb C$. Note that the translation by the vector $q_{x x'} (v)$ does not change the diameter of the set, we obtain
  \[ \text{diam} \{ q_{x y}(v) : y \in U \} = \text{diam} \{ q_{x' y}(D  h_{x x'}^s (v)) : y \in U \}. \]
\end{proof}

\section{Proof of Theorem \ref{t1}}
In \cite{ghy95}, Ghys proved that the stable holonomy for holomorphic Anosov diffeomorphisms is a Möbius transformation when $\dim_{\mathbb{C}} E^u=1$ by computing the Schwarzian derivative (see also \cite{ghy93} for the case when $\dim_{\mathbb{R}} E^u=1$). We regard $q_{xy}$ in Lemma \ref{l1} as a substitution for the Schwarzian derivative and generalize Ghys' argument to the quasiconformal case, showing that the holonomy maps are Möbius transformations up to a linear coordinate change. The global product structure is then an immediate corollary.

We first review some definitions and introduce notations related to the Markov partition. Let $f$ be an Anosov diffeomorphism of $M$. For every $\varepsilon > 0$, there exists a $\eta > 0$ such that if two points, $x$ and $y$, in $M$ are at a distance less than $\eta$, then the discs within the unstable leaves $W^u(x)$ and stable leaves $W^s(y)$ with centers at $x$ and $y$, and radii $\varepsilon$, intersect at a single point denoted as $[x, y]$.

A rectangle is a compact connected set $R$ in $M$, which is the closure of its interior, with a diameter less than $\eta$, and such that if $x$ and $y$ are in $R$, then $[x, y]$ is also in $R$. If $R$ is a rectangle and $x \in R$, we define:
\[ R_u (x) = \{ [x, y], y \in R \} , \quad R_s (x) = \{ [y, x], y \in R \}.\]
A Markov partition for $f$ is a finite collection of rectangles $R^1, \ldots, R^N$ covering $M$, with disjoint interiors, satisfying the following property. If $x$ belongs to the interior of $R^i$ and $f (x)$ belongs to the interior of $R^j$, then
\[ R_u^j (f (x)) \subset f (R_u^i (x)), \quad R_s^i (x) \subset f^{- 1} (R_s^j (f (x)))\]
We can always assume that the boundaries of $R^i_u$ and $R^i_s$ have zero Lebesgue measure (see \cite{man87}) and here we do not need the transitivity assumption (see \cite{bow70} and \cite{sin68} for existence and properties of Markov partition). 

For each rectangle $R^i$, we choose a base point $x_i$ in the interior of
$R^i$ and we shall simply write $R^i_u$ and $R^i_s$ instead of $R^i_u (x_i)$
and $R^i_s (x_i)$ respectively. Sometimes, when we do not want to specify to
which rectangle $R^i$ a point $x$ belongs, we write $R_u (x)$ and $R_s (x)$
instead of $R^i_u (x)$ and $R^i_s (x)$ respectively.

Let $\mathcal{U}$ be the disjoint union of the $R^i_u$ for $i = 1, \ldots, N$. Let $x \in \mathcal{U}$. If $f (x)$ belongs to the interior of $R^j$, we set $\phi (x) = [x_j, f (x)]$. In this way, we obtain a partially defined expanding map $\phi$ from an open dense subset of $\mathcal{U}$ to $\mathcal{U}$. We shall not attempt to define $\phi (x)$ for those points $x$ such that $f (x)$ belongs to two rectangles. However, we note the following. For each $i = 1, \ldots, N$, there is a finite collection of diffeomorphisms $\psi_1^i, \ldots, \psi_{N_i}^i$ from $R^i_u$ to $\mathcal{U}$ which are branches of the inverse of $\phi$, i.e., such that if $\phi (x) = y$ is defined and $x$ is in $\mathcal{U}$, then $x$ is one of the points $\psi^i_1 (y), \ldots, \psi_{N_i}^i (y)$. If the Markov partition is sufficiently thin, one can always assume that for each $i$, the image of the $\psi^i_1, \ldots, \psi_{N_i}^i$ is contained in different $R^j_u$.

After these preliminaries, we can begin the proof of Theorem \ref{t1}. 

Let $v$ be a vector tangent to $\mathcal{U}$, thought of as a vector tangent to $W^u$ at some point $x$ in some $R^i_u$. For fixed $v$, we consider the continuous map defined on $R_s^i (x)$
\[ y \in R_s^i (x) \mapsto q_{x y} (v) \in \mathbb{C}, \]
where $q_{xy}$ is defined in the previous section. Let $\delta (v) = \text{diam} \{ q_{x y} (v) : y \in R^i_s \}$. Recall that $q_{x y}$ is defined by the Schwazian derivative, which is a holomorphic quadratic differential, so we have
\[\delta (\lambda v) =  \lambda ^2 \delta (v)\]
for any $\lambda \in \mathbb{R}$. Hence, $\delta$ has the tensorial character of an area form, i.e., a 2-form of type (1,1). More precisely, let $z$ be a holomorphic coordinate in a local unstable leaf $W^u_{\text{loc}}(x)$, then the 2-form $\mu: = \delta(\frac{\partial}{\partial z}) d z d \bar z$  does not depend on the choice of this parameter. In this way, $\mu$ can be regarded as a natural measure on $\mathcal{U}$.

\begin{lemma}
  \label{inv}The measure $\mu$ is invariant under $\phi$.
\end{lemma}

\begin{proof}
  Let $v$ be a vector tangent to $\mathcal{U}$ at a point $x$ belonging to
  $R^i_u$. Consider the $f^{- 1}$-image of $R_s^i (x)$. It is the union of
  $R_s (\psi_1^i (x)), \ldots, R_s (\psi_{N_i}^i (x))$. By Lemma \ref{l2} and
  Lemma \ref{l3}, we have:
  \begin{eqnarray*}
    \delta (v) & = & \text{diam} \{ q_{x y} (v) : y \in R^i_s (x) \}\\
    & = & \text{diam} \left\{ q_{f^{- 1} x z} \left. \left( \left( D
    f|_{E^u_{f^{- 1} x}} \right)^{- 1} (v) \right) : z \in f^{- 1} (R^i_s (x))
    \right\} \right.\\
    & \leqslant & \sum_{k = 1}^{N_i} \text{diam} \left\{ q_{f^{- 1} x z}
    \left( \left( D f|_{E^u_{f^{- 1} x}} \right)^{- 1} (v) \right) : z \in
    R_s^i (\psi_k^i (x)) \right\}\\
    & = & \sum_{k = 1}^{N_i} \text{diam} \left\{ q_{\psi_k^i (x) z} \left( D
    h_{f^{- 1} x \psi_k^i (x)}^s \circ \left( D f|_{E^u_{f^{- 1} x}}
    \right)^{- 1} (v) \right) : z \in R_s^i (\psi_k^i (x)) \right\}\\
    & = & \sum_{k = 1}^{N_i} \text{diam} \{ q_{\psi_k^i (x) z} ((D
    \psi_k^i)^{- 1} (v)) : z \in R_s^i (\psi_k^i (x)) \}\\
    & = & \sum_{k = 1}^{N_i} \delta ((D \psi_k^i)^{- 1} (v))
  \end{eqnarray*}
  Therefore, $\mu$ is a sub-invariant measure, i.e., for every Borel set $\mathcal{B} \subset \mathcal{U}$, one has
  \[ \mu (\mathcal{B}) \leqslant \mu (\phi^{- 1} \mathcal{B}) . \]
  It is now easy to deduce that $\mu$ is invariant. If $\mathcal{B}$ is a
  Borel set in $\mathcal{U}$, one has
  \[ \mu (\mathcal{U}) = \mu (\mathcal{B}) + \mu
     (\mathcal{U}\backslash\mathcal{B}) \leqslant \mu (\phi^{- 1}
     \mathcal{B}) + \mu (\phi^{- 1} (\mathcal{U}\backslash\mathcal{B})) =
     \mu (\phi^{- 1} \mathcal{U}) = \mu (\mathcal{U}) \]
  Hence, all inequalities are equalities and $\mu$ is an invariant measure.
\end{proof}

\begin{lemma}
  \label{nw}
  If $\mu$ is not the zero measure, then $\delta$ is nowhere vanishing.
\end{lemma}

\begin{proof}
  We prove by contradiction. If $\delta$ vanishes at $x \in R_u^i$, then
  \[ \text{diam} \{ q_{x y} (v) : y \in R^i_s (x) \} = 0 \]
  for every $v \in E^u_x$. Since all inequalities in Lemma \ref{inv} are
  equalities, by taking an iteration of $f$, we have
  \[ 0 = \text{diam} \{ q_{x y} (v) : y \in R^i_s (x) \} = \sum_{\phi^n (x') =
     x} \text{diam} \{ q_{x' z} \nobracket ((D \phi^n_{x'})^{- 1} (v)) : z \in
     R_s (x') \} \]
  for any $n \in \mathbb{N}, v \in E^u_{x'}$. Since $D \phi^n_{x'} : E^u_{x'}
  \rightarrow E^u_x$ is a linear automorphism of vector spaces, we have
  \[ \text{diam} \{ q_{x' z} \nobracket (v) : z \in R_s (x') \} \nobracket = 0
  \]
  for any $x' \in \phi^{- n} (x)$. Denote $A \assign \{ x' \in \mathcal{U}|
  \phi^n (x') = x \text{ for some } n \in \mathbb{N} \}$, then $\delta$
  vanishes on $A$. Since $f$ is topologically transitive and $\cup_{n=1}^\infty f^{-n}(R_s^i(x)) \subset A$, $A$ is then a dense
  subset of $\mathcal{U}$. Note that $q_{x y}$ and $\text{diam} \{ q_{x y} (v)
  : y \in R^i_s (x) \}$ varies continuously as $x$ changes, $\delta$ vanishes
  on $A$ implies that $\delta$ vanishes on $\mathcal{U}$, which implies that $\mu$ is the zero measure. 
\end{proof}

The following combinatorial facts is used in \cite{ghy95}:

\medskip

\noindent \textbf{Fact 1.} \textit{Let $K$ be a connected compact set in $\mathbb{C}$ which is the union of finitely many nonempty compact sets $K_j$. If the diameter of $K$ is the sum of the diameters of $K_j$ and no $K_j$ is a single point set, then no point of $K$ belongs to three distinct $K_j$.}

\medskip

\noindent \textbf{Fact 2.} \textit{A compact set $C$ in $\mathbb{R}^p
(p \geqslant 2)$ with nonempty interior cannot be covered by finitely many
compact sets $C_j$ with arbitrarily small diameters and such that no point of
$C$ belongs to three distinct $C_j$.}

\medskip

\begin{lemma}\label{lemma mu vanish}
  $\mu$ is the zero measure and $\delta$ vanishes everywhere. 
\end{lemma}
\begin{proof}
 We prove by contradiction. For any $x \in \mathcal{U}$, $v \in E^u_x$, it follows from the definitions that
    \begin{eqnarray*}
    \{q_{x y} (v) : y \in R_s (x)\} & = &  \{q_{f^{- n} x z} (D f^{-n}_x(v)) : z \in f^{- n} (R_s (x))\} \\
    & = & \bigcup_{\phi^n
     (x') = x} \{ q_{f^{- n} x z} (D f^{-n}_x(v)) : z \in R_s
     (x')) \} 
  \end{eqnarray*}
Since all inequalities in Lemma \ref{inv} are equalities, we also have
  \[ \text{diam} \{q_{x y} (v) : y \in R_s (x)\} =
     \sum_{\phi^n
     (x') = x}\text{diam}\{ q_{f^{- n} x z} (D f^{-n}_x(v)) : z \in R_s
     (x')). \]
 Now if $\mu$ is not the zero measure, then by Lemma \ref{nw}, $\delta$ is nowhere vanishing and no $\{ q_{f^{- n} x z} (D f^{-n}_x(v)) : z \in R_s
     (x')) \} $ is a single point set. By fact 1, no point of $ \{q_{x y} (v) : y \in R_s (x)\}$ belongs to three distinct compact sets among $\{ q_{f^{- n} x z} (D f^{-n}_x(v)) : z \in R_s
     (x')) \} $. Let $n$ tend to infinity, we get covers of $\{q_{x y} (v) : y \in R_s (x)\}$ by finitely many compact sets $\{ q_{f^{- n} x z} (D f^{-n}_x(v)) : z \in R_s
     (x')) \}$ with arbitrarily small diameters, which contradicts Fact 2. Therefore, $\mu$ is the zero measure and $\delta$ vanishes everywhere. 
\end{proof}

\begin{proposition}
  \label{pj}
  For any $x \in M, y \in W^s (x)$, $A_y \circ \Phi_y \circ h^s_{x y} \circ \Phi_x^{- 1} \circ A_x^{- 1}$ is a Möbius transformation.
\end{proposition}

\begin{proof}
  By Lemma \ref{lemma mu vanish}, for any $x \in R^i_u$, $v \in E^u_x$, we have
  \[ \text{diam} \{ q_{x y} (v) : y \in R^i_s (x) \} = 0. \]
  Since $q_{x x} (v) = S (\text{Id}) (v) = 0$, $q_{x y} = 0$ for every $y \in R^i_s (x)$. Therefore, $q_{x y} = 0$ in every rectangle $R^i_u$. Since all rectangles form an open dense subset of $M$ and $q_{x y}$ varies continuously as $x, y$ changes, we see that $q_{x y} = 0$ holds on the ambient manifold $M$, and
  \[ A_y \circ \Phi_y \circ h^s_{x y} \circ \Phi_x^{- 1} \circ A_x^{- 1} \in \text{PSL}(2, \mathbb{C}) \]
  is a Möbius transformation.
\end{proof}
Theorem \ref{t1} follows immediately from the following proposition, which is essentially proved in \cite{ks03}.
\begin{proposition}[Proof of Proposition 3.3 in \cite{ks03}]\label{ksb}
   Let $f$ be a transitive $C^{\infty}$ uniformly $u$-quasiconformal Anosov diffeomorphism of a compact manifold $M$. If $A_y \circ \Phi_y \circ h^s_{x y} \circ \Phi_x^{-1} \circ A_x^{-1}$ is a Möbius transformation for any $x \in M, y \in W^s(x)$, then the stable holonomy is globally defined. 
\end{proposition}
Combining Proposition \ref{pj} and Proposition \ref{ksb}, we see that the stable holonomy is globally defined when $\dim E^u = 2$. This finishes the proof of Theorem \ref{t1}.
\bibliographystyle{plain}
\bibliography{ref}

\end{document}